\documentclass{amsart}
\usepackage{graphicx}
\usepackage{amsmath}
\usepackage{amssymb}
\usepackage[dvipsnames]{xcolor}
\usepackage{amsthm}
\usepackage{thmtools}
\usepackage{thm-restate}
\usepackage{amsfonts}
\usepackage{graphicx}
\usepackage{stackrel}
\usepackage[backend = biber]{biblatex}
\usepackage[hyperindex,colorlinks,breaklinks]{hyperref}
\usepackage[hyphenbreaks]{breakurl}
\usepackage[mathscr]{eucal}
\usepackage[shortlabels]{enumitem}
\usepackage{tikz}
\usepackage{tikz-cd}

\newtheorem{theorem}{Theorem}

\numberwithin{theorem}{section}
\numberwithin{equation}{section}
\newtheorem{definition}[theorem]{Definition}
\newtheorem{lemma}[theorem]{Lemma}

\newtheorem{proposition}[theorem]{Proposition}

\newtheorem{corollary}[theorem]{Corollary}

\newcommand{\Z}{\mathbb{Z}}
\newcommand{\pZ}{\widehat{\Z}}
\newcommand{\Ic}{\mathfrak{I}}
\newcommand{\Jc}{\mathfrak{Y}}
\newcommand{\hM}{\widehat{M}}
\newcommand{\hG}{\widehat{\Gamma}}
\newcommand{\hN}{\widehat{N}}
\newcommand{\hD}{\widehat{\Delta}}
\newcommand{\hL}{\widehat{\Lambda}}
\newcommand{\Gs}[2][\Lambda]{\mathcal{G}_{#1}(#2)}
\newcommand{\Ann}[2][\Lambda]{\operatorname{Ann}_{#1}(#2)}

\newcommand{\at}[1]{|_{#1}}
\newcommand{\fp}{\mathfrak{p}}
\newcommand{\fm}{\mathfrak{I}}
\newcommand{\LI}{\Lambda_{\Ic}}
\newcommand{\ad}{_{\widehat{\Ic}}}

\title{Profinite Rigidity over Noetherian Domains}
\author{Julian Wykowski}
\address{Department of Pure Mathematics and Mathematical Statistics, Centre for Mathematical Sciences, Wilberforce Road, CB3 0WB, United Kingdom}
\email{jw2006@cam.ac.uk}
\date{\today}

\hypersetup{
    breaklinks=true,
	linkcolor={ForestGreen},
	citecolor={ForestGreen},
	urlcolor={ForestGreen}
}

\begin{document}
\begin{abstract}
    We initiate the study of profinite rigidity for modules over a Noetherian domain: to what extent are these objects determined by their finite images? We establish foundational statements in analogy to classical results in the category of groups. We describe three profinite invariants of modules over any Noetherian domain $\Lambda$. We show that free modules are profinitely rigid when $\Lambda$ satisfies a homological condition, and characterise the profinite genus of all modules when $\Lambda$ is a Dedekind domain. In the case where $\Lambda$ is a PID, we find that all finitely generated modules are profinitely rigid. As an application, we prove that solvable Baumslag--Solitar groups are profinitely rigid in the absolute sense. These are the first examples of absolute profinite rigidity among non-abelian one-relator groups and among non-LERF groups.
\end{abstract}
\maketitle

\section{Introduction}
Can the entirety of an infinite object be seen in its finite shadows? This question, known as \emph{profinite rigidity}, is a major theme in various mathematical theories. Group theory studies to what extent properties of a finitely generated residually finite group $\Gamma$ can be detected in the collection of its finite quotients, or equivalently, in its \emph{profinite completion} $\hG$, the latter being a topological group arising as the inverse limit of the finite quotients of $\Gamma$. The central problem of this enterprise is \emph{absolute profinite rigidity}, i.e. the question whether the isomorphism type of the profinite completion $\hG$ determines the isomorphism type of $\Gamma$ uniquely among all finitely generated residually finite groups. Despite considerable research (see \cite{Reid_Survey} or \cite{Bridson_survey} for a survey), the absolute question remains vastly unanswered; for free groups, it is a notorious open problem attributed to Remeslennikov \cite[Question 15]{Remeslennikov}. To the author's knowledge, the only currently known examples of absolute profinite rigidity are: abelian and certain virtually polycyclic groups \cite{Pickel1,Pickel2,Pickel3}, affine Coxeter groups \cite{Coxeter1,Coxeter2}, lamplighter groups whose base has prime exponent \cite{blacha}, and, thanks to a recent breakthrough, certain Kleinian groups \cite{Bridson_et_al, Bridson_et_al_2,Bridson_et_al_3,Cheetham-West}. The latter class are the only known examples among full-sized groups. In this article, we expand the list with an infinite family of solvable groups.
\begin{restatable}{theoremA}{thmB}\label{Thm::Groups}
    Let $n \in \mathbb{Z}$ be an integer. The solvable Baumslag--Solitar group ${\Gamma = \operatorname{BS}(1,n)} \cong \langle a,t \mid t^{-1}at = a^n\rangle $ is profinitely rigid in the absolute sense.
\end{restatable}
To the author's knowledge, Theorem~\ref{Thm::Groups} yields the first examples of absolute profinite rigidity among non-abelian one-relator groups, as well as the first known examples among groups which are not subgroup separable (LERF). We note also that Theorem~\ref{Thm::Groups} contrasts with \cite{Pickel_Metabelian}, where infinite families of non-isomorphic metabelian groups with isomorphic profinite completions were constructed, and \cite{Nikolov_Segal_2}, where uncountable such families of solvable groups were found.

The strategy of proof for profinite rigidity results typically involves the reduction of statements about the finite quotients of a group $\Gamma$ to statements about the finite quotients of certain modules arising as subgroups or representation of $\Gamma$. However, so far, these arguments have been constrained to the specific groups under consideration, due to the absence of a general theory of profinite rigidity for modules over rings $\Lambda \neq \mathbb{Z}$. In this article, we initiate the research of profinite rigidity in the generality of modules over \emph{any} Noetherian domain. This approach paves the way for a more structured enterprise of profinite rigidity research, yielding applications such as Theorem~\ref{Thm::Groups} and further work of the author published in parallel \cite{FreeMetab}.

Given a ring $\Lambda$ and a $\Lambda$-module $M$, the \emph{profinite completion} of $M$ in the category of $\Lambda$-modules is the topological $\Lambda$-module $\hM$ given by the inverse limit of the finite $\Lambda$-epimorphic images of $M$ endowed with the profinite topology, i.e. the topology induced from the product of discrete topologies on all finite quotients. We shall restrict ourselves to the study of \emph{residually finite} $\Lambda$-modules---that is, modules $M$ such that the canonical map $\iota \colon M \to \hM$ is an injection---since any profinite property of a $\Lambda$-module $M$ factors through its maximal residually finite quotient $M/\operatorname{Ker}(\iota)$. We aim to investigate when the following definition is satisfied.
\begin{definition}\label{Def::APR}
     Let $\Lambda$ be a commutative ring. A finitely generated residually finite $\Lambda$-module $M$ is \emph{$\Lambda$-profinitely rigid in the absolute sense} if the equivalence \[
     \widehat{M} \cong \widehat{N} \Longleftrightarrow M \cong N
     \]
    holds for any finitely generated residually finite $\Lambda$-module $N$.
\end{definition}

 Similarly, we say that a property, quantity or a functor $\mathcal{P}(-)$ of $\Lambda$-modules is a \emph{$\Lambda$-profinite invariant} if 
 $
 \hM \cong \hN
 $
 implies
 $\mathcal{P}(M) = \mathcal{P}(N)
 $ whenever $M,N$ are finitely generated residually finite $\Lambda$-modules. In the absence of absolute profinite rigidity, one may instead ask to exhibit certain $\Lambda$-profinite invariants or determine the set of isomorphism classes of $\Lambda$-modules whose profinite completion is isomorphic to that of a given $\Lambda$-module $M$, i.e. the \emph{profinite genus} of $M$.

 \begin{definition}
Let $\Lambda$ be a commutative ring and $M$ be a finitely generated residually finite $\Lambda$-module. The \emph{$\Lambda$-profinite genus} $\Gs{M}$ of $M$ is the set of isomorphism classes of finitely generated residually finite $\Lambda$-modules $N$ admitting an isomorphism of profinite completions $\hM \cong \hN$.     
 \end{definition}
 
 The focus of the present article is to establish various $\Lambda$-profinite invariants in the generality where $\Lambda$ is any Noetherian domain, and absolute $\Lambda$-profinite rigidity in the presence of further assumptions on $\Lambda$. To begin, we assume only that $\Lambda$ is finitely generated as a $\Z$-algebra; in that setting, all finitely generated modules are residually finite (see Lemma~\ref{Lem::NB}). Nonetheless, the diversity of modules over Noetherian domains endows the study of profinite rigidity in these categories with complexity and consequence. In Section~\ref{Sec::Gen}, we estblish a series of results linking the profinite completion of a $\Lambda$-module to its adic completions, allowing us to pass between local and global methods in certain contexts. In Section~\ref{Sec::PI}, we investigate to what extent a finitely generated $\Lambda$-module $M$ is determined by its $\Lambda$-profinite completion $\hM$. For instance, we describe contexts in which the profinite annihilator $\Ann[\hL]{\hM}$ coincides with the closure of the discrete annihilator $\overline{\Ann[\Lambda]{M}}$ in the profinite ring $\hL$: see Proposition~\ref{Prop::AnnAndComp}, an essential ingredient in the proof of Theorem~\ref{Thm::Groups}. In greater generality, this seemingly subtle distinction turns out intimately related to the study of integral closure for ideals in regular Noetherian rings, itself an active area of research in commutative algebra and complex geometry \cite{Sznajdman}.
 
 Going further, we establish three $\Lambda$-profinite invariants for modules over any finitely generated Noetherian domain $\Lambda$, culminating in the following result.
\begin{restatable}{propositionA}{propA}\label{Prop::General}
        Let $\Lambda$ be a finitely generated Noetherian domain and $M,N$ be finitely generated $\Lambda$-modules with $\hM \cong \hN$. Then:
    \begin{enumerate}
        \item The annihilators of $M$ and $N$ in $\Lambda$ agree: $\Ann{M} = \Ann{N}$.
        \item For any maximal ideal $\fm$, the minimal number of generators of the localisations $M_\fm$ and $N_\fm$ as $\Lambda_\Ic$-modules agree: $\mu_\fm(M) = \mu_\fm(N)$.
        \item The $\Lambda$-module $M$ is projective of rank $\nu \in \mathbb{N}$ if and only if the $\Lambda$-module $N$ is projective of rank $\nu$.
    \end{enumerate}
\end{restatable}

We say that a commutative ring $\Lambda$ is \emph{homologically taut} if finitely generated projective modules are free. In that case, Proposition~\ref{Prop::General} yields the following profinite rigidity result for free modules over a maximally general class of domains.

\begin{restatable}{theoremA}{thmA}\label{Thm::epic}
    Let $\Lambda$ be a homologically taut finitely generated Noetherian domain. Finitely generated free $\Lambda$-modules are $\Lambda$-profinitely rigid in the absolute sense.
\end{restatable}
In the opposite direction, one may ask to what extent profinite rigidity holds for modules over Noetherian domains $\Lambda$ where projectivity is equilvalent to absence of torsion. This property characterises the class of \emph{Dedekind domains}, that is, integrally closed Noetherian domains of dimension one. In that context, a strong structure theorem for finitely generated $\Lambda$-modules in terms of the ideal classes of $\Lambda$ allows us to determine precisely the obstruction to profinite rigidity for finitely generated $\Lambda$-modules. We dedicate Section~\ref{Sec::PIDs} to the proof of the following result.

\begin{restatable}{theoremA}{thmC}\label{Thm::Dedekind}
Let $\Lambda$ be a finitely generated Dedekind domain. For any finitely generated $\Lambda$-module $M$, there exists an injective assignment of sets \[
\mathcal{Y}_{\Lambda,M} \colon \Gs{M} \hookrightarrow \operatorname{Cl}(\Lambda)
\] mapping the $\Lambda$-profinite genus of $M$ into the ideal class group of $\Lambda$.
\end{restatable}

Note that Dedekind domains with trivial ideal class group are precisely principal ideal domains (PIDs). In that case, we obtain profinite rigidity in the maximal generality of \emph{all} finitely generated $\Lambda$-modules.

\begin{restatable}{corollaryA}{corB}\label{Cor::PID}
     Let $\Lambda$ be a finitely generated principal ideal domain. Any finitely generated $\Lambda$-module is $\Lambda$-profinitely rigid in the absolute sense.
\end{restatable}

Finally, in Section~\ref{Sec::Gps}, we develop applications to the profinite rigidity of groups. The strategy is to reduce the question of profinite rigidity of a finitely generated group $\Gamma$ to the question of profinite rigidity for a suitably chosen modular represtentation of $\Gamma$. This strategy is converse to \cite{Pickel_Metabelian}, where non-isomorphic modules with isomorphic finite quotients were used to produce counterexamples to profinite rigidity among metabelian groups. The novelty in the present case lies in the leveraging of the full theory of profinite modules together with modern algebraic machinery to obtain the positive profinite rigidity result in Theorem~\ref{Thm::Groups}. Specifically, we utilise the results developed in Sections \ref{Sec::Gen}, \ref{Sec::PI} and \ref{Sec::PIDs} in conjunction with a celebrated result of Quillen and Suslin \cite{quillen,suslin} on projective modules over polynomial rings, to reduce the proof of Theorem~\ref{Thm::Groups} to a number-theoretic statement. The latter turns out equivalent to a question of Erd\H{o}s which has been resolved in \cite{corrales_schoof}.

In a parallel article \cite{FreeMetab}, the author has also applied the results of the present paper to prove that finitely generated free metabelian groups are profinitely rigid in the absolute sense. Ongoing investigations into the profinite rigidity of additional classes of solvable groups---in particular, torsion-free lamplighter groups---are due to appear in future work.

\section*{Acknowledgements}
The author is grateful to his PhD supervisor, Gareth Wilkes, for enlightening discussions and advice throughout the project. Thanks are also due to Will Cohen, Francesco Fournier-Facio, Tim Santens, Henrique Souza and Henry Wilton for helpful conversations and acquaintance with various references. Financially, the author was supported by a Cambridge International Trust \& King's College Scholarship.

\section{General Theory}\label{Sec::Gen}
Let us begin with the establishment of a few foundational results regarding profinite completions of modules over a commutative ring $\Lambda$ in analogy to classical statements regarding profinite completions and the profinite rigidity of groups. We invite the reader to \cite[Section 5.5]{RZ} and \cite[Section 6.1]{Gareth_Book} for the necessary background regarding profinite completions in categories of groups, rings and modules. We shall also relate the theory of profinite modules to the theory of adic completions; for an introduction on the latter, see \cite[Chapter III]{bourbaki}.

\subsection{Foundational results}

Let $\Lambda$ be a commutative ring and $M$ be a $\Lambda$-module. We say that the ring $\Lambda$ is \emph{finitely generated} if there exists an epimorphism of $\Z$-algebras $\Z[x_1, \ldots, x_d] \twoheadrightarrow \Lambda$ for some $d \in \mathbb{N}$. Similarly, we say that the $\Lambda$-module $M$ is \emph{finitely generated} if there exists an epimorphism of $\Lambda$-modules $\Lambda^\mu \twoheadrightarrow M$ for some $\mu \in \mathbb{N}$; in that case, we write $\mu(M)$ to denote the minimal number $\mu$ which admits such an epimorphism, i.e. the minimal number of generators of $M$. We shall say that $M$ is \emph{finite} if its underlying set is finite---this aligns with group theoretic terminology but should not be confused with some terminology in commutative algebra where a finite module can mean what we define as finitely generated. Finally, we say that a $\Lambda$-module $M$ is \emph{residually finite} if for any $m \in M - 0$ there exists an epimorphism of $\Lambda$-modules $\pi \colon M \to Q$ onto a finite $\Lambda$-module $Q$ such that $\pi(m) \neq 0$ holds. Note that finitely generated commutative rings are automatically Noetherian by the Hilbert Basis Theorem. Nonetheless, we shall continue referring to these commutative rings as \emph{finitely generated Noetherian} rings to highlight the indispensability of the Noetherian property. However, thanks to the following result, we may omit the assumption of residual finiteness for finitely generated modules over finitely generated Noetherian rings.
\begin{lemma}[Theorem 1 in \cite{Nut_Berry}] \label{Lem::NB}
    Let $\Lambda$ be a finitely generated Noetherian domain. Any finitely generated $\Lambda$-module is residually finite. 
\end{lemma}
Note that this contrasts sharply with the category of groups, where there are plenty of finitely generated constructions which are not residually finite: for instance, certain Baumslag--Solitar groups (see Section~\ref{Sec::BS}).

We shall write $\mathcal{C}_\Lambda(M)$ for the collection of isomorphism classes of finite $\Lambda$-epimorphic images of $M$. The profinite completion $\hM$ of $M$ is then identified canonically with the structure of a topological $\Lambda$-module given by the inverse limit
\[
\hM = \lim_{\substack{\longleftarrow \\ Q \in \mathcal{C}_\Lambda(M)}} Q
\]
where $\mathcal{C}_\Lambda(M)$ forms an inverse system under projections. Considering the commutative ring $\Lambda$ as a module over itself, the profinite completion $\hL$ is identified with the inverse limit of quotients of $\Lambda$ by finite-index ideals. Let $\iota \colon \Lambda \to \hL$ be the natural profinite completion map. Any profinite $\hL$-module $M$ is naturally endowed with the structure of a $\Lambda$-module with action via $\iota$. In this way, we obtain a forgetful functor
\[
F \colon \mathbf{PfMod}(\hL) \to \mathbf{PfMod}(\Lambda)
\] from the category of profinite $\hL$-modules to the category of profinite $\Lambda$-modules. Conversely, let $M$ be a profinite $\Lambda$-module which is topologically finitely generated. By \cite[Lemma 2]{Lewin}, the number of open submodules of a given index $\nu \in \mathbb{N}$ is finite, so the submodule
\[
 M(\nu) = \bigcap_{\substack{K \triangleleft M\\ [M:K] \leq \nu}}{K} \trianglelefteq M
\]
must be open in $M$ as well. The finite abelian groups $M / M(\nu)$ have finite endomorphism rings for all $\nu \in \mathbb{N}$, so $\operatorname{End}(M)$ acquires naturally the structure of a profinite ring given by the inverse limit
\[
\operatorname{End}(M) = \lim_{\substack{\longleftarrow \\ \nu \in \mathbb{N}}} \operatorname{End}\left(\frac{M}{M(\nu)}\right)
\]
over the partially ordered set $\mathbb{N}$. It follows that the action $\Lambda \to \operatorname{End}(M)$ factors through the profinite completion $\iota \colon \Lambda \to \hL$, yielding a commutative diagram
\[
\begin{tikzcd}
    \hL \arrow[r] & \operatorname{End}(M) \\
    \Lambda \arrow[u, "\iota"] \arrow[ur, bend right = 15] & 
\end{tikzcd}
\]
via the universal property of the profinite completion. In this way, one obtains a profinite $\hL$-module structure $M'$ on $M$ and the restriction of the $\hL$-action on $M'$ to the dense subring $\Lambda$ agrees with the $\Lambda$-action on $M$, giving $F(M') = M$. Hence, the functor $F$ is essentially surjective when restricted to the category of (topologically) finitely generated profinite modules. Similarly, one finds that this restricted functor is faithful and full, i.e. it yields a bijective correspondence between continuous $\hL$-morphisms $M \to N$ and continuous $\Lambda$-modrphisms $M \to N$ whenever $M,N$ are (topologically) finitely generated profinite $\hL$-modules. It follows that $F$ is in fact an equivalence of categories, which we record here for completeness.

\begin{lemma}
    Let $\Lambda$ be a commutative ring and $\hL$ its profinite completion. The forgetful functor
    \[
    F \colon \mathbf{PfMod_{fg}}(\hL) \to \mathbf{PfMod_{fg}}(\Lambda)
    \]
    is an equivalence between the category of (topologically) finitely generated profinite $\hL$-modules and the category of (topologically) finitely generated profinite $\Lambda$-modules.
\end{lemma}

For this reason, it is immaterial whether $\hM$ is considered as a module over the discrete ring $\Lambda$ or the profinite ring $\hL$; we choose the former in the interest of simplicity. The fundamental significance of profinite completions in the category of groups is that their isomorphism type encompasses precisely the information contained in the finite epimorphic images of a finitely generated group. Analogously, we observe the following foundational result which motivates the study of profinite rigidity for modules over a commutative ring.

\begin{theorem}\label{Thm::FiniteQuotients}
 Let $\Lambda$ be a finitely generated Noetherian ring and $M,N$ be finitely generated $\Lambda$-modules. Then $\hM \cong \hN$ if and only if $\mathcal{C}_\Lambda(M) = \mathcal{C}_\Lambda(N)$.
\end{theorem}

The proof follows \textit{mutatis mutandis} from the proof of \cite[Main Theorem]{classic} or \cite[Theorem 3.2.3]{Gareth_Book}; it is omitted here in the interest of brevity. We obtain the following corollary in analogy to \cite[Theorem 2]{classic}, supplanting the proof with an argument regarding the number of subobjects of a given index.
\begin{theorem}\label{Thm::FiniteQuotientsEpi}
 Let $\Lambda$ be a finitely generated commutative ring and $M,N$ be finitely generated residually finite $\Lambda$-modules with $\hM \cong \hN$. Any epimorphism $f \colon M \twoheadrightarrow N$ is an isomorphism.
\end{theorem}
\begin{proof}
    Note first that for any positive integer $\nu \in \mathbb{N}$, the number of index-$\nu$ submodules of the finitely generated $\Lambda$-module $M$ is finite: this is established for finitely generated rings in \cite[Lemma 2]{Lewin} and follows to the present case by pulling back submodules to ideals in the free $\Z$-algebra on $d_\Lambda(M) \cdot d_\Z(\Lambda)$ generators. As above, it follows that the submodules
    \[
    M(\nu) = \bigcap_{\substack{K \triangleleft M\\ [M:K] \leq \nu}}{K} \trianglelefteq M\qquad \text{and} \qquad N(\nu) = \bigcap_{\substack{K \triangleleft N\\ [N:K] \leq \nu}}{K} \trianglelefteq N
    \]
    have finite index in $M$ and $N$, respectively. Now let $f \colon M \twoheadrightarrow N$ be an epimorphism of $\Lambda$-modules and assume that $\hM \cong \hN$. Suppose for a contradiction that $f$ is not injective, so there exists a non-trivial element of the kernel $m \in \operatorname{Ker}(f) - 0$. As $M$ is residually finite, there must exist a submodule $M' \triangleleft M$ of $M$ of finite index $\nu \in \mathbb{N}$ such that $m \notin M' \supseteq M(\nu)$. Since the index of the epimorphic image of an index-$\nu$ submodule must divide $\nu$, we obtain a commutative diagram
    \[
    \begin{tikzcd}
        M \arrow[r, "f", two heads] \arrow[d, two heads, "\pi_\nu"] & N  \arrow[d, two heads, "\pi_\nu"]\\
        M/M(\nu) \arrow[r, "f_\nu", two heads] & N/N(\nu)
    \end{tikzcd}
    \] of $\Lambda$-morphisms. Then $\widetilde{x} = \pi_\nu(x) \neq 0$ has $f(\widetilde{x}) = 0$ by construction. However, Theorem~\ref{Thm::FiniteQuotients} yields $M/M(\nu) \cong N/N(\nu)$, as these are the unique maximal quotients of $M$ and $N$, respectively, wherein the intersection of submodules of index at most $\nu$ is trivial. Hence $f_\nu \colon M/M(\nu) \to N/N(\nu)$ is a surjective but not injective map between finite sets of equal cardinality, a violation of the pigeonhole principle.
\end{proof}

\subsection{Relationship with $\Ic$-adic completions}

In the second part of this section, we shall explore the relationship between profinite completions and \emph{adic completions} studied in commutative algebra.
These results will prove essential for the analysis in Sections \ref{Sec::PI} and \ref{Sec::BS}. However, due to the technical nature of these results, we recommend to skip this subsection on a first reading of the article. Recall that the localisation of a commutative ring $\Lambda$ at a prime ideal $\Ic$ is the commutative ring obtained from $\Lambda$ by adjoining inverses to all elements in $\Lambda - \Ic$. Given a $\Lambda$-module $M$, we may form the localised module $M_\Ic = M \otimes_\Lambda \LI$ accordingly; we invite the reader to \cite[Chapter II]{bourbaki} for details. Modules over a Noetherian ring $\Lambda$ localised at a maximal ideal $\Ic \trianglelefteq \Lambda$ are equipped naturally with a filtration given by
\[
M \supseteq \Ic M \supseteq \Ic^2 M \supseteq \ldots \supseteq \bigcap_{j \in \mathbb{N}} \Ic^j M = 0
\]
and thus embed into the Hausdorff completion $M_{\widehat{\fm}}$ with respect to this filtration, that is, the so-called $\fm$-adic completion. We invite the reader to \cite[Chapter III]{bourbaki} for details on this construction. Note that the $\Ic$-adic topology on $M_\Ic$ will not in general be a profinite topology: for instance, the quotient $\Lambda / \Ic$ could yield the characteristic zero field of rational numbers $\mathbb{Q}$ which has no finite quotients. However, in the case where $\Lambda$ is finitely generated, the field $\Lambda/\Ic$ must be finite \cite[Exercise 5.25]{AM} and the $\Ic$-adic completions $\Lambda\ad$ and $M\ad$ are indeed profinite. In fact, one obtains the following description of the profinite completion $\hL$ in terms of its adic completions.

\begin{lemma}[Lemma 2.1 in \cite{jaikin_genus}]\label{Lem::CRT}
    Let $\Lambda$ be a finitely generated commutative ring and $\hL$ its profinite completion. Write $\operatorname{Max}(\Lambda)$ to denote the set of maximal ideals of $\Ic$. There is an isomorphism of profinite rings
    \[
    \hL \cong \prod_{\Ic \in \operatorname{Max}(\Lambda)} \Lambda_{\widehat\Ic}
    \]
    where $\Lambda_{\widehat\Ic}$ denotes the $\Ic$-adic completion of $\Lambda$.
\end{lemma}

A particular consequence of Lemma~\ref{Lem::CRT} is the following description of torsion in the profinite completion of a finitely generated commutative ring $\Lambda$. Recall that a Noetherian local ring $\Lambda$ is \emph{regular local} if the minimal number of generators of its maximal ideal is equal to the dimension of $\Lambda$. A Noetherian ring $\Lambda$ is \emph{regular} if its localisation $\Lambda_\Ic$ is regular local for each maximal ideal $\Ic \trianglelefteq \Lambda$.

\begin{lemma}\label{Lem::hLtf}
    Let $\Lambda$ be a finitely generated regular commutative ring. The profinite completion $\hL$ is torsion-free as a $\Lambda$-module. 
\end{lemma}
\begin{proof}
    Note first that a finitely generated commutative ring is automatically Noetherian by the Hilbert Basis Theorem. By Lemma~\ref{Lem::NB}, the finitely generated commutative ring $\Lambda$ embeds into its profinite completion $\hL$ via the canonical map $\iota \colon \Lambda \to \hL$, and we identify $\Lambda$ with its image $\iota(\Lambda)$. The assertion that $\hL$ is torsion free as a $\Lambda$-module is then equivalent to the statement that no element of $\Lambda$ embeds into $\hL$ as a zero-divisor. Under the additional assumption that $\Lambda$ is regular, its localisation $\Lambda_\Ic$ at any maximal ideal $\Ic \trianglelefteq \Lambda$ is regular local, whence so is the $\Ic$-adic completion $\Lambda\ad$ \cite[Exercise 19.1]{eisenbud}. But regular local rings are in particular domains \cite[Corollary 10.14]{eisenbud}, so it follows via Lemma~\ref{Lem::CRT} that $\hL$ is a product of domains indexed over $\Ic \in \operatorname{Max}(\Lambda)$. Moreover, the composition
    \[
    \begin{tikzcd}
\Lambda \arrow[r, "\iota"] \arrow[rd, two heads, bend right = 25] & \hL \cong \prod_{\Ic \in \operatorname{Max}(\Lambda)} \Lambda_{\widehat\Ic} \arrow[d, "\operatorname{proj}_\Ic"] \\
& \Lambda\ad                         
\end{tikzcd}
    \]
    is precisely the $\Ic$-adic completion map, which is injective by \cite[Chapter III Proposition 3.3.5]{bourbaki}. It then follows that for any $\lambda \in \Lambda - 0$, the image $\iota(\lambda)$ is non-zero at each coordinate in the product of domains $\hL \cong \prod_{\Ic \in \operatorname{Max}(\Lambda)} \Lambda\ad$, and in particular, $\iota(\lambda)$ is not a divisor of zero in the profinite ring $\hL$. Hence the profinite completion $\hL$ is indeed torsion-free as a $\Lambda$-module.
\end{proof}

For the remainder of this section, we shall develop results allowing us to relate profinite and adic completions of $\Lambda$-modules to one another. We commence with the following elementary lemma, which is likely known to experts but which we include here for completeness.

\begin{lemma}\label{Lem::AdicsAndCompletions}
    Let $\Lambda$ be a finitely generated Noetherian domain and $M$ be a finitely generated $\Lambda$-module. For any maximal ideal $\Ic \trianglelefteq \Lambda$, there exist isomorphisms of $\Lambda$-modules
    \[
    M_{\widehat\Ic} \cong M \otimes_{\Lambda} \Lambda_{\widehat\Ic} \cong \hM \,\, \widehat{\otimes}_{\hL} \,\, \Lambda_{\widehat\Ic}
    \]
    where $M_{\widehat\Ic}$ denotes the $\Ic$-adic completion of $M$, while $- \widehat{\otimes}_{\hL} - $ denotes the completed tensor product over $\hL$ and $\hM$ denotes the $\Lambda$-profinite completion of $M$.
\end{lemma}

\begin{proof}
    The first isomorphism is given in \cite[pp. 203]{bourbaki}. For the second, note that $M_\Ic$ is defined as the inverse limit
    \begin{equation}\label{Eq::InvLim1}
    M_{\widehat\Ic} = \lim_{\leftarrow k \in \mathbb{N}} \left( \frac{M}{\Ic^kM}\right) 
    \end{equation}
    whereas $\hM \,\, \widehat{\otimes}_{\hL} \,\, \Lambda_{\widehat\Ic}$ is given by the inverse limit
    \begin{equation}\label{Eq::InvLim2}
    \hM \,\, \widehat{\otimes}_{\hL} \,\, \Lambda_{\widehat\Ic} = \lim_{\leftarrow k \in \mathbb{N}, N \trianglelefteq_f M} \left( \frac{M}{N} \otimes_\Lambda \frac{\Lambda}{\Ic^k}\right)
    \end{equation}
    where $k \in \mathbb{N}$ and $N\trianglelefteq_f M$ ranges through finite-index $\Lambda$-submodules of $M$. Note that for any $k \in \mathbb{N}$, there is an isomorphism
    \[
    \frac{M}{\Ic^k M} \cong \frac{M}{\Ic^k M} \otimes_\Lambda \frac{\Lambda}{\Ic^k}
    \]
    so that the inverse system in (\ref{Eq::InvLim1}) embeds into the inverse system in (\ref{Eq::InvLim2}). On the other hand, for any $k \in \mathbb{N}$ and $N \trianglelefteq_f M $, the surjection
    \[
    \frac{M}{\Ic^k M} \cong M \otimes_\Lambda \frac{\Lambda}{\Ic^k} \twoheadrightarrow \frac{M}{N} \otimes_\Lambda \frac{\Lambda}{\Ic^k}
    \]
    witnesses that the inverse system in (\ref{Eq::InvLim1}) has cofinal image embedded into the inverse system in (\ref{Eq::InvLim2}). Thus one obtains an isomorphism of inverse limits and the result. 
\end{proof}

A natural question is whether the localisation of the profinite completion of a $\Lambda$-module at the closure of a maximal ideal $\Ic$ returns the $\Ic$-adic completion. The following result answers this question to the affirmative.

\begin{proposition}\label{Prop::LocAndComp}
    Let $\Lambda$ be a finitely generated Noetherian domain, let $M$ be a $\Lambda$-module and write $\hM$ to denote its profinite completion. For any maximal ideal $\Ic \trianglelefteq \Lambda$, the localisation $\hM_{\overline{\Ic}}$ of the profinite $\hL$-module $\hM$ at the open maximal ideal $\overline{\Ic} \trianglelefteq \hL$ is isomorphic to the $\Ic$-adic completion $M\ad$ of $M$ as a $\hL$-module.
\end{proposition}
\begin{proof}
    It will suffice to show that the result holds in the case $M = \Lambda$; the general case $M \neq \Lambda$ then follows from
    \[
    \hM_{\overline{\Ic}} \cong \hM \widehat\otimes_{\hL} {\hL}_{\overline{\Ic}} \cong \hM \widehat\otimes_{\hL} \Lambda\ad \cong M\ad
    \]
    where the first isomorphism is given in \cite[Proposition 3.5]{AM} and the third isomorphism is given in Lemma~\ref{Lem::AdicsAndCompletions}. By Lemma~\ref{Lem::CRT}, there is an epimorphism of profinite $\hL$-algebras $p_\Ic \colon \hL \twoheadrightarrow \Lambda\ad$ given by projection. We claim that $p_\Ic(\hL - \overline\Ic) \subseteq \hL^\times$ holds. Indeed, if $\lambda \in \hL - \overline\Ic$ then the image of $p_\Ic(\lambda)$ in the finite field $\hL/ \overline{\Ic} \cong \Lambda/\Ic$ is a non-zero element and hence a unit. It follows that the polynomial $p_\Ic(\lambda) \cdot  x - 1$ in $\Lambda\ad[x]$ has a solution modulo $\overline{\Ic}$. By Hensel's Lemma \cite[Theorem 7.3]{eisenbud}, the polynomial $p_\Ic(\lambda) \cdot x -1$ must then also have a solution in $\Lambda\ad$, so in fact $p_\Ic(\lambda)$ is a unit, as claimed. The universal property of localisations then tells us that $p_\Ic$ factors through the localisation map $\hL \to \hL_{\overline{\Ic}}$, yielding a commutative diagram
    \[
    \begin{tikzcd}
        \hL \cong \prod_{\operatorname{Max}(\Lambda)} \Lambda\ad \arrow[r, two heads, "p_\Ic"] \arrow[d, "j"] & \Lambda\ad \\
         \hL_{\overline{\Ic}} \arrow[ur, two heads, "\vartheta"', bend right = 15] &
    \end{tikzcd}
    \]
    for some epimorphism of rings $\vartheta$. Moreover, the local ring $\hL_{\overline{\Ic}}$ is profinite: it is compact Hausdorff since it is the quotient of the compact space $\hL \times (\hL - \overline{\Ic})$ by a closed equivalence relation, and thus profinite via \cite[Proposition 5.1.2]{RZ}. Hence it will suffice to show that for any finite $\Lambda$-epimorphic quotient
    $
    \pi \colon \hL_{\overline{\Ic}} \twoheadrightarrow Q
    $
    there exists a finite $\Lambda$-epimorphism $\varpi \colon \hL \twoheadrightarrow \Lambda / \Ic^k$ for some $k \in \mathbb{N}$ satisfying $\varpi \circ \vartheta = \pi$. Indeed, if the $\Lambda$-module $Q$ is finite then in particular it is Artinian and its Jacobson radical is nilpotent \cite[pp. 89]{AM}. Hence there exists $k \in \mathbb{N}$ such that $\pi(\overline\Ic)^k = 0$, or equivalently,
    \[
    \overline{\Ic}^k \subseteq \operatorname{Ker}(\pi \circ j \colon \hL \twoheadrightarrow Q)
    \]
    holds. This means in particular that the epimorphism $\pi \circ j$ factors through the quotient $\hL \twoheadrightarrow \Lambda / \Ic^k$, which in turn factors through the $\Ic$-adic completion $\Lambda\ad$, yielding a morphism $\varpi \colon \Lambda\ad \to \Lambda / \Ic^k$ which satisfies $\varpi \circ \vartheta = \pi$. Thus $\vartheta$ must be an isomorphism of rings and by construction also an isomorphism of $\hL$-modules. This completes the proof.
\end{proof}

A particular consequence of Proposition~\ref{Prop::LocAndComp} is that for any two modules $M$ and $N$ over a finitely generated Noetherian domains $\Lambda$ with isomorphic profinite completions $\hM \cong \hN$, the localisations $M_\Ic$ and $N_\Ic$ at a maximal ideal $\Ic \trianglelefteq \Lambda$ also have isomorphic profinite completions $\widehat{M_\Ic} \cong \widehat{N_\Ic}$.
The following proposition makes no reference to the finite generation of the Noetherian domain $\Lambda$. In fact, it turns out that this implication holds in greater generality, even without the assumption of finite generation on the ring $\Lambda$.

\begin{proposition}\label{Prop::Localisation}
    Let $\Lambda$ be a Noetherian domain, $\Ic \trianglelefteq \Lambda$ be a prime ideal and $M,N$ be $\Lambda$-modules. An isomorphism of profinite $\Lambda$-modules $\hM \cong \hN$ induces an isomorphism of profinite $\Lambda_\Ic$-modules $\widehat{M_\Ic} \cong \widehat{N_\Ic}$.
\end{proposition}

\begin{proof}
    Let $\Lambda$ be a commutative ring, $\Ic \trianglelefteq \Lambda$ be a prime ideal and $M,N$ be $\Lambda$-modules. Write $\LI$ for the localisation of $\Lambda$ at $\Ic$ and $M_\Ic = M \otimes_{\Lambda} \LI$ for the localisation of $M$ at $\Ic$. We claim that a finite $\Lambda$-module $Q$ has finite localisation $Q_\Ic = Q \otimes_{\Lambda} \LI$. Indeed, given $m \in Q$, the annihilator $\Ann{m} \trianglelefteq \Lambda$ is the kernel of the map $\varphi_m \colon \Lambda \to M$ given by $\varphi_m \colon \lambda \mapsto \lambda \cdot m$, so in particular it forms a finite index ideal. On the other hand, the relation $m \otimes \lambda^{-1}_1 = m \otimes \lambda^{-1}_2$ in $Q_\Ic$ is equivalent to $\lambda_1 - \lambda_2 \in \Ann{m}$ for any $\lambda \in \Lambda - \Ic$. It follows that for each $m \in Q$, there are finitely many elements of the form $m \otimes \lambda^{-1}$ in $Q_\Ic$. But $Q$ itself is finite, and every element of the localisation $Q_\Ic$ is of the form $m \otimes \lambda^{-1}$ for some $m \in Q$ and $\lambda \in \Lambda - \Ic$. Hence the localised module $Q_\Ic = Q \otimes_\Lambda \LI$ is finite, as claimed. Consider now the collection
    \[
    \mathcal{L}_{\Ic}(M) = \{Q \otimes_\Lambda \LI : Q \in \mathcal{C}_\Lambda(M)\}
    \]
    which forms an inverse system of finite $\LI$-modules with partial order inherited from $\mathcal{C}_\Lambda(M)$. As localisation is exact, any epimorphism of $\Lambda$-modules $M \twoheadrightarrow Q$ induces an epimorphism of $\LI$-modules $M \otimes_\Lambda \LI \twoheadrightarrow Q \otimes_\Lambda \LI$. Thus we obtain an assignment
    \[
    \mathcal{L}_{\Ic}(M) \to \mathcal{C}_{\LI}(M_\Ic)
    \]
    given by the identity. We claim that the image of this assignment is a cofinal subsystem of $\mathcal{C}_{\LI}(M_\Ic)$, i.e. that for any epimorphism of $\LI$-modules $\pi \colon M_\Ic \to T$ with finite codomain, there exists an epimorphism $\varpi \colon M \to Q$ such that $\pi$ factors through the map
    \[
    \begin{tikzcd}
        M \otimes_\Lambda \LI \arrow[r, two heads, "\varpi \otimes 1"] & Q \otimes_\Lambda \LI
    \end{tikzcd}
    \]
    induced via localisation. Indeed, consider the composition
    \[
    M \xrightarrow{f} M \otimes_\Lambda \LI \xrightarrow{\pi} T
    \]
    where $f$ is the canonical map associated to the localisation. Let $K = \operatorname{Ker}(\pi f)$ be $\Lambda$-submodule of $M$ given by the kernel of this composition. The short exact sequence of the quotient $M/K$ induces a short exact sequence of localised modules, yielding a commutative diagram
    \[
    \begin{tikzcd}
        0 \arrow[r] & K \arrow[r] \arrow[d] & M \arrow[r,"\varpi"] \arrow[d, "f"] & M/K \arrow[r] \arrow[d] & 0 \\
        0 \arrow[r] & K \otimes_\Lambda \LI \arrow[r] & M \otimes_\Lambda \LI \arrow[r, "\varpi \otimes 1"] \arrow[d, "\pi", two heads] & M/K \otimes_\Lambda \LI \arrow[r] \arrow[ld, "\vartheta", two heads, dashed] & 0 \\
        && T&&
    \end{tikzcd}
    \]
    wherein the map $\vartheta$ is induced from $\pi$ via the universal property of the quotient: if $k \otimes \lambda^{-1} \in K \otimes_\Lambda \LI$ then
    \[
    \pi(k \otimes \lambda^{-1}) = \lambda^{-1} \cdot \pi(k \otimes 1) = \lambda^{-1} \cdot  \pi(f(k)) = 0
    \]
    so indeed $\pi$ factors through the surjection and $\vartheta$ exists. The quotient morphism of $\Lambda$-modules $\varpi \colon M \to M/K$ then induces a map the localised morphism of $\LI$-modules $\varpi \otimes 1$ satisfying $\pi = \vartheta \circ (\varpi \otimes 1)$. This proves the claim that $\mathcal{L}_\Ic(M)$ is cofinal in $\mathcal{C}_{\LI}(M_\Ic)$. The analogous argument for the $\Lambda$-module $N$ shows that $\mathcal{L}_\Ic(N)$ is cofinal in $\mathcal{C}_{\LI}(N_\Ic)$ as well. Hence there exists an isomorphism of $\Lambda$-modules
    \[
    \widehat{M_\Ic} \cong  \lim_{Q \in \mathcal{C}_\Lambda(M)} (Q \otimes_\Lambda \LI) \cong \lim_{Q \in \mathcal{C}_\Lambda(N)} (Q \otimes_\Lambda \LI) \cong \widehat{N_\Ic}
    \]
    where the central isomorphism holds as $\hM \cong \hN \Leftrightarrow \mathcal{C}_\Lambda(M) \cong \mathcal{C}_\Lambda(N)$. As the action of $\Lambda$ on $\widehat{M_\Ic}$ and $\widehat{N_\Ic}$ factors through the action by $\LI$, we obtain an isomorphism of profinite $\LI$-modules $\widehat{M_\Ic} \cong \widehat{N_\Ic}$, as postulated.
\end{proof}

\section{Profinite Invariants}\label{Sec::PI}
In this section, we prove Proposition~\ref{Prop::General}. We record the first point---our first profinite invariant of $\Lambda$-modules---in the following result, which holds without any further assumptions on the commutative ring $\Lambda$. 

\begin{lemma}\label{Lem::Ann}
    Let $\Lambda$ be a commutative ring and $M,N$ be residually finite $\Lambda$-modules. If $\hM \cong \hN$ then $\Ann{M} = \Ann{N}$.
\end{lemma}

\begin{proof}
    Let $\Lambda$ be a commutative ring and $M,N$ be $\Lambda$-modules with an isomorphism of $\Lambda$-profinite completions $\hM \cong \hN$. Choose any $\lambda \in \Ann{M}$. Then $\lambda$ acts as the zero endomorphism on $M$, and consequently also on all finite $\Lambda$-epimorphic images thereof. By Theorem~\ref{Thm::FiniteQuotients}, it must also act as zero on all finite $\Lambda$-epimorphic images of $N$. If there was an $n \in N$ with $\lambda n \neq 0$ then we could pass to a finite $\Lambda$-quotient $\pi \colon N \twoheadrightarrow Q$ whereby $\lambda\cdot n$ is mapped to a non-trivial element, since $N$ is residually finite by assumption. But then $\lambda \cdot \pi(n) = \pi(\lambda \cdot n) \neq 0$, a contradiction. Hence $\lambda \cdot N = 0$ and $\lambda \in \Ann{N}$. We conclude that $\Ann{M} \subseteq \Ann{N}$ and note that the opposite inclusion holds symmetrically.
\end{proof}
Now, isomorphism classes of cyclic $\Lambda$-modules (that is, modules which are generated over $\Lambda$ by a single element) are determined entirely by their annihilators, so we obtain the following corollary to Lemma~\ref{Lem::Ann}.
\begin{corollary}\label{Cor::Cyclic}
    Let $\Lambda$ be a commutative ring and $M,N$ be residually finite cyclic $\Lambda$-modules. If $\hM \cong \hN$ then $M \cong N$.
\end{corollary}

Another natural question one may ask is whether given a $\Lambda$-module $M$, the closure of the annihilator $\Ann[\Lambda]{M}$ in $\hL$ is equal to the annihilator $\Ann[\hL]{\hM}$ of the profinite completion $\hM$. In general, this turns out to be a hard question. However, in the case where the profinite completion $\hM$ is cyclic as a profinite $\hL$-module (note that this assumption is weaker than the assumption that $M$ itself is cyclic), we obtain the following affirmative answer. It will be essential in the proof of Theorem~\ref{Thm::Groups} in Section~\ref{Sec::BS}.

\begin{proposition}\label{Prop::AnnAndComp}
    Let $\Lambda$ be a finitely generated Noetherian domain and $M$ be a finitely generated $\Lambda$-module whose profinite completion $\hM$ forms a cyclic $\hL$-module. Then
    \[
    \overline{\Ann{M}} = \Ann[\hL]{\hM}
    \]
    where $\overline{\Ann{M}}$ denotes the closure of the ideal $\Ann{M}$ in $\hL$ and $\Ann[\hL]{\hM}$ denotes the annihilator of the profinite module $\hM$ in the profinite ring $\hL$.
\end{proposition}

\begin{proof}
    Let $A = \Ann{M} \trianglelefteq \Lambda$ be the annihilator of the $\Lambda$-module $M$ identified with its image in $\hL$ and let $\overline{A}$ be its closure in $\hL$. On the other hand, write $\widetilde{A} = \Ann[\hL]{\hM}$ to denote the annihilator of the profinite $\hL$-module $\hM$. As $A \subseteq \widetilde{A}$ and $\widetilde{A}$ is closed, the inclusion $\overline{A} \subseteq \widetilde{A}$ must hold. Conversely, suppose for a contradiction that there exists $a \in \widetilde{A} - \overline{A}$. Then there exists a finite $\Lambda$-epimorphic image $\pi \colon \hL \to \Theta$ with $\pi(a) \notin \pi(A)$. Furthermore, by \cite[pp. 88]{bourbaki},  there exists a maximal ideal $\widetilde{\Ic} \trianglelefteq \Theta$ such that
    \begin{equation}\label{Eq::LocalAnnImage}
    \pi(a)_{\widetilde{\Ic}} \notin \pi(A)_{\widetilde{\Ic}}
    \end{equation}
    holds in the localised ring $\Theta_{\widetilde{\Ic}}$. Consider hence the maximal ideal $\Ic = \pi^{-1}(\widetilde{\Ic}) \cap \Lambda$ of $\Lambda$, which fits into the commutative diagram
    \[
    \begin{tikzcd}
        \Lambda \arrow[r, "\pi \at{\Lambda}"] \arrow[d, "(\cdot)_{\Ic}"] & \Theta \arrow[d, "(\cdot)_{\widetilde{\Ic}}"] \\
        \Lambda_\Ic \arrow[r, "\pi_{\Ic}"]                               & \Theta_{\widetilde{\Ic}}                    
\end{tikzcd}
    \] 
    as per \cite[pp. 69]{bourbaki}. As $\Theta$ is finite, it is Artinian and so its Jacobson radical is nilpotent \cite[pp. 89]{AM}, and in particular there exists an integer $k \in \mathbb{N}$ such that $\widetilde{\Ic}^k = 0$, id est, $\Ic^k \subseteq \operatorname{Ker}(\pi)$. Note that the $\Lambda_\Ic$-module $M_\Ic = M \otimes_\Lambda \Lambda_\Ic$ is cyclic: by assumption, the profinite $\hL$-module $\hM$ is cyclic, hence so is its quotient $\frac{M}{\Ic M} \cong \frac{M_\Ic}{\Ic M_\Ic}$ and by Nakayama's Lemma it follows that $M_\Ic$ must also be cyclic. Now $\Lambda$ is dense in $\hL$ and $\overline{\Ic^k}$ is open, so we may produce $\lambda \in \Lambda$ satisfying $\lambda - a \in \overline{\Ic^k}$. We then find
    \[
    \lambda_{\Ic} \cdot \frac{M_\Ic}{\Ic^k M_\Ic} = \lambda \cdot \frac{M}{\Ic^k M} = a \cdot \frac{M}{\Ic^k M} = 0
    \]
    since $a \in \widetilde{A}$ and the action of $\hL$ on the finite quotient module $\frac{M}{\Ic^k M}$ factors through the $\Lambda$-profinite completion $\hM$ of $M$. On the other hand,
    \[
    \frac{M_\Ic}{\Ic^k M_\Ic} \cong \frac{\Lambda_\Ic}{\Ic^k + \Ann[\Lambda_\Ic]{M_\Ic}} \cong \frac{\Lambda_\Ic}{\Ic^k + A_\Ic}
    \]
    where the first isomorphism exists as $M_\Ic$ is a cyclic $\Lambda_\Ic$-module and the second isomorphism derives from \cite[pp. 68]{bourbaki}. The conjunction of these equations yields \[\lambda_\Ic \in \operatorname{Ann}_{\Lambda_\Ic} \left( \frac{\Lambda_\Ic}{\Ic^k + A_\Ic}\right) = \Ic^k + A_\Ic
    \]
    wherefore
    \[
    \pi(a)_{\widetilde{\Ic}} = \pi(\lambda)_{\widetilde{\Ic}} = \pi_\Ic(\lambda_\Ic) \in \pi_\Ic(A_\Ic) = \pi(A)_{\widetilde{\Ic}}
    \]
    which contradicts (\ref{Eq::LocalAnnImage}). Thus $\widetilde{A} = \overline{A}$ and the result is established.
\end{proof}

We proceed to demonstrate the main result of the present section, Proposition~\ref{Prop::General}, wherein we establish three key profinite invariants for modules over a finitely generated Noetherian domain.

\propA*
\begin{proof}
    That finitely generated modules over a finitely generated Noetherian domain are residually finite is given in Lemma~\ref{Lem::NB}. We proceed to show that the items (1)-(3) are $\Lambda$-profinite invariants. Item \textbf{(1)} is given in Lemma~\ref{Lem::Ann}.
    
    For \textbf{(2)}, choose any maximal ideal $\fm \trianglelefteq \Lambda$ and  write $\mathbb{F}$ to denote the field $\mathbb{F} \cong \Lambda / \fm$, which must be finite due to the assumption that $\Lambda$ is finitely generated as a $\Z$-algebra (see \cite[Exercise 5.25]{AM}). Let $\mu_{\fm}(-) \colon \Lambda\text{-mod} \to \mathbb{N}$ denote the minimal number of generators of the localisation $(-) \otimes_\Lambda \LI$ of a module $(-)$ over the local ring $\Lambda_\fm$. We need to show that for any finitely generated $\Lambda$-module $N$ satisfying $\hM \cong \hN$, the equation $\mu_\fm(M) = \mu_\fm(N)$ holds. We claim first that the equation
    \begin{equation}\label{Eq::SupDim}
        \mu_\fm(M_\fm) = d_\fm(M_\fm) := \sup\{\rho \in \mathbb{N} \mid \mathbb{F}^\rho \in \mathcal{C}_{\Lambda_\fm}(M_\fm)\}
    \end{equation}
    holds. Indeed, Nakayama's Lemma \cite[Section II.3.2 Proposition 5]{bourbaki} yields
    \begin{equation}\label{Eq::Nakayama}
        \mu_\fm(M_\fm) = \dim_{\Lambda/\fm}(M_\fm/\fm M_\fm)
    \end{equation}
    so in particular $\mu_\fm(M_\fm) \leq d_\fm(M_\fm)$ holds. For the opposite inequality, suppose that there is an epimorphism of $\Lambda_\fm$-modules
    $
    \varpi \colon M_\fm \twoheadrightarrow \mathbb{F}^\rho
    $
    for some $\rho \in \mathbb{N}$. As the tensor product $- \otimes_{\Lambda_\fm} \Lambda_\fm / \fm \Lambda_\fm$ is exact, we obtain an epimorphism of $\Lambda_\fm$-modules
    \[
     M_\fm/\fm M_\fm \cong M_\fm \otimes_\Lambda \Lambda/\fm \twoheadrightarrow \mathbb{F}^\rho \otimes_\Lambda \Lambda/\fm \cong \mathbb{F}^\rho
    \]
    so in particular $\rho \leq \dim_{\Lambda/\fm}(M_\fm/\fm M_\fm)$ via rank-nullity. Another application of (\ref{Eq::Nakayama}) then yields the claim and (\ref{Eq::SupDim}) is established. The analogous argument for the $\Lambda_\fm$-module $N_\fm$ yields $\mu_\fm(N_\fm) = d_\fm(N_\fm)$. Now $\hM \cong \hN$ implies $\widehat{M_\fm} \cong \widehat{N_\fm}$ and hence $\mathcal{C}_{\Lambda_\fm}(M_\fm) = \mathcal{C}_{\Lambda_\fm}(N_\fm)$ via Lemma~\ref{Prop::Localisation} and Theorem~\ref{Thm::FiniteQuotients}. But $d_\fm(-)$ is parametrised entirely by $C_\fm(-)$, so we obtain
    \[
    \mu_\fm(M_\fm) = d_\fm(M_\fm) = d_\fm(N_\fm) = \mu_\fm(N_\fm)
    \]
    as postulated. Thus item (2) is established.
    
    Finally, we shall demonstrate that item \textbf{(3)} is a $\Lambda$-profinite invariant. Suppose that the finitely generated $\Lambda$-module $M$ is projective of rank $\nu \in \mathbb{N}$ and $N$ is a finitely generated $\Lambda$-module satisfying $\hM \cong \hN$. We aim to show that $N$ must also be projective of rank $\nu$. By \cite[Section II.5.3 Theorem 2]{bourbaki}, it will suffice to show that for any maximal ideal $\Ic \trianglelefteq \Lambda$, the localisation $N_\fm$ at $\fm$ is a free $\Lambda_\fm$-module of rank $\nu$. Choose any maximal ideal $\Ic \trianglelefteq \Lambda$ and consider the local ring $\Lambda_\fm$. By Lemma~\ref{Prop::Localisation}, there is an isomorphism of $\Lambda_\fm$-modules $\widehat{M_\fm} \cong \widehat{N_\fm}$. But \cite[Section II.5.3 Theorem 2]{bourbaki}, tells us that the $\Lambda_\fm$-module $M_\fm$ is must be free of rank $\nu$. As we have shown in item (2) above, it follows that $\mu_\fm(N_\fm) = \mu_\fm(M_\fm) = \nu$, so in particular, there exists an epimorphism of $\Lambda_\fm$-modules
    \[
    M_\fm \cong \Lambda_\fm^\nu \twoheadrightarrow N_\fm
    \]
    given by the final term in a free resolution of $N_\fm$. Now Theorem~\ref{Thm::FiniteQuotientsEpi} yields $N_\fm \cong M_\fm$, as required. Hence $N$ must be projective of rank $\nu$ and the proof is complete.
\end{proof}
Under the additional assumption that $\Lambda$ is homologically taut, there is only one isomorphism class of projective $\Lambda$-modules of a given rank. In that case, Proposition~\ref{Prop::General}.3 tells us that any other module in the profinite genus of a free module must also belong to that isomorphism class. Thus we obtain the following first profinite rigidity result for modules over homologically taut Noetherian domains.

\thmA*

\section{Dedekind Domains}\label{Sec::PIDs}
In this section, we focus on the specific case where $\Lambda$ is a Dedekind domain, that is, an integrally closed Noetherian domain of dimension one. In that case, we can bound the profinite genus $\Gs[\Lambda]{M}$ of a $\Lambda$-module $M$ in terms of the ideal class group $\operatorname{Cl}(\Lambda)$, see Theorem~\ref{Thm::Dedekind} proven below.  We invite the reader to \cite[Section II.5.7]{bourbaki} for the definition and basic properties of the ideal class group $\operatorname{Cl}(\Lambda)$ of a Dedekind domain $\Lambda$. In the case where $\Lambda$ is a principal ideal domain (PID), the ideal class group is trivial, so we find that \emph{all} finitely generated $\Lambda$-modules are profinitely rigid in the absolute sense: see Corollary~\ref{Cor::PID} below.

Indeed, suppose that $\Lambda$ is a finitely generated Dedekind domain and let $M$ be a finitely generated $\Lambda$-module. Write $T_\Lambda(M)$ to denote the torsion submodule of $M$, that is,\[
T_\Lambda(M) = \{m \in M \mid \exists \lambda \in \Lambda - 0, \lambda \cdot m = 0\}
\] with $T_\Lambda(M) \trianglelefteq M$. By the structure theorem of finitely generated modules over Dedekind domains \cite[Chapter VII Section 4.10]{bourbaki}, the $\Lambda$-module $M$ decomposes as a direct sum
\begin{equation}\label{Eq::Struct}
M = T_\Lambda(M) \oplus P_M
\end{equation}
where $P_M$ is some projective $\Lambda$-module. The torsion module $T_\Lambda(M)$ itself decomposes as a direct sum of cyclic modules
\begin{equation}\label{Eq::TorsionModule}
T_\Lambda(M) = \bigoplus_{i=1}^l \frac{\Lambda}{\fp_i^{k_i}}
\end{equation}
for some tuple of prime ideals $\fp_i \trianglelefteq \Lambda$ and positive integers $k_i \in \mathbb{N}$ indexed by $i = 1, \ldots, l$. Now prime ideals in a Dedekind domain are maximal and maximal ideals in a finitely generated $\Z$-algebra must have finite index \cite[Exercise 5.25]{AM}. Thus each summand in (\ref{Eq::TorsionModule}) must in fact be a finite $\Lambda$-module. We have proven:
\begin{lemma}\label{Lem::TorsionFinite}
    Finitely generated torsion modules over a finitely generated Dedekind domain must be finite. 
\end{lemma}
In particular, any finitely generated torsion $\Lambda$-module is canonically isomorphic to its profinite completion. Returning to the finitely generated $\Lambda$-supermodule $M$, the injectivity of the canonical $\Lambda$-profinite completion map $\iota_M \colon M \to \hM$ now implies that we may identify the torsion module $T_\Lambda(M)$ with its isomorphic image $\iota_M(T_\Lambda(M))$ in $\hM$. It follows via \cite[Theorem 6.1.7]{Gareth_Book} that
\begin{equation}\label{Eq::StructProf}
\hM = T_\Lambda(M) \oplus \widehat{P_M}
\end{equation}
is a decomposition of profinite $\Lambda$-modules. We claim that the image $T_\Lambda(M)$ of the torsion module of $M$ agrees with the torsion submodule $T_\Lambda(\hM)$ of the profinite $\Lambda$-module $\hM$. The inclusion $T_\Lambda(M) \subseteq T_\Lambda(\hM)$ follows directly from the $\Lambda$-invariance of $\iota_M$. Conversely, suppose that $m \in \hM$ is a $\Lambda$-torsion element of the profinite $\Lambda$-module $\hM$ and consider its image $\overline{m}$ under the projection $\hM \twoheadrightarrow \widehat{P_M}$, which forms a $\Lambda$-torsion element of $\widehat{P_M}$. As $P_M$ is a finitely generated projective $\Lambda$-module, it embeds as a direct summand into a finitely generated free module $\Lambda^\mu$, so the profinite completion $\widehat{P_M}$ embeds into the profinite completion $\hL^\mu$ as a direct summand as well \cite[Theorem 6.1.7]{Gareth_Book}. The latter is torsion-free as a $\Lambda$-module via the conjunction of Lemma~\ref{Lem::hLtf} and the fact that Dedekind domains are regular \cite[Chapter VII Theorem 2.2.1]{bourbaki}, so the $\Lambda$-module $\widehat{P_M}$ must also be free of $\Lambda$-torsion. It follows that $\overline{m} = 0$ and $m \in \operatorname{Ker}(\hM \twoheadrightarrow \widehat{P_M}) = T_\Lambda(M)$ must have been an element of $T_\Lambda(M)$ to begin with. We have proven the following.
\begin{lemma}
    Let $\Lambda$ be a finitely generated Dedekind domain and let $M$ be a finitely generated $\Lambda$-module. The canonical $\Lambda$-profinite completion morphism $\iota_M \colon M \to \hM$ restricts to an isomorphism
    \[
    T_\Lambda (M) \xrightarrow{\sim} T_\Lambda (\hM)
    \]
    of torsion $\Lambda$-submodules.
\end{lemma}
We are now in a position to prove the characterisation of profinite genera for modules over a Dedekind domain in terms of its ideal classes.

\thmC*

\begin{proof}
    Let $\Lambda$ be a finitely generated Dedekind domain of characteristic zero and let $\operatorname{Cl}(\Lambda)$ denote its ideal class group. Let $M$ be a finitely generated $\Lambda$-module and write $d(M) = \dim_{\operatorname{Frac}(\Lambda)}\left(M \otimes_\Lambda \operatorname{Frac(M)}\right)$ for the dimension of its localisation at the prime ideal $(0)$ as a vector space over the fraction field. We shall construct an injective assignment of sets
    \begin{equation}\label{Eq::koten}
        \mathcal{Y}_{\Lambda,M} \colon \Gs{M} \hookrightarrow \operatorname{Cl}(\Lambda)
    \end{equation}
    as follows. Choose any $N \in \Gs{M}$, so that $N$ is a finitely generated $\Lambda$-module admitting an isomorphism of $\Lambda$-profinite completions $f \colon \hM \to \hN$. By the structure theorem of finitely generated modules over a Dedekind domain \cite[Chapter VII Section 4.10]{bourbaki}, the $\Lambda$-modules $N$ and $M$ decompose as a direct sums
    \[
        M = T_\Lambda(M) \oplus P_M \qquad \text{and} \qquad N = T_\Lambda(N) \oplus P_N
    \]
    where $T_\Lambda(M)$ and $T_\Lambda(N)$ are the respective torsion submodules, while  $P_M$ and $P_N$ are projective modules and $P_M$ has projective rank $d(M)$. By Lemma~\ref{Lem::TorsionFinite}, the canonical $\Lambda$-profinite completion morphisms
    \[
    \iota_M \colon M \to \hM \qquad \text{and} \qquad \iota_N \colon N \to \hN
    \] restrict to isomorphisms
    \[
    T_\Lambda (M) \xrightarrow{\sim} T_\Lambda (\hM) \qquad \text{and} \qquad T_\Lambda (N) \xrightarrow{\sim} T_\Lambda(\hN)
    \]
    of the respective torsion submodules. Moreover, the $\Lambda$-invariant isomorphism of profinite modules $f \colon \hM \xrightarrow{\sim} \hN$ must satisfy $f(T_\Lambda(\hM)) = T_\Lambda(\hN)$ owing to the fact that torsion submodules are characteristic. Hence we obtain a commutative diagram
    \[
    \begin{tikzcd}
0 \arrow[r] & T_\Lambda(M) \arrow[r] \arrow[d, "\sim"']    & M \arrow[r] \arrow[d, "\iota_M"]      & P_M \arrow[r] \arrow[d, "\iota_{P_M}"]     & 0 \\
0 \arrow[r] & T_\Lambda(\hM) \arrow[r] \arrow[d, "\sim"'] & \hM \arrow[r] \arrow[d, "f"]         & \widehat{P_M} \arrow[r] \arrow[d, "\sim"] & 0 \\
0 \arrow[r] & T_\Lambda(\hN) \arrow[r]                    & \hN \arrow[r]                         & \widehat{P_N} \arrow[r]                    & 0 \\
0 \arrow[r] & T_\Lambda(N) \arrow[r] \arrow[u, "\sim"]    & N \arrow[r] \arrow[u, "\iota_N"'] & P_N \arrow[r] \arrow[u, "\iota_{P_N}"']     & 0
\end{tikzcd}
    \]
    where the top right and bottom right vertical maps agree with the profinite completion map by \cite[Theorem 6.1.7]{Gareth_Book}. Hence we obtain isomorphisms
    \[
    T_\Lambda(M) \cong T_\Lambda(N) \qquad \text{and} \qquad \widehat{P_M} \cong \widehat{P_N}
    \]
    as $\Lambda$-modules. Now $P_M$ is projective of rank $d(M)$, so $P_N$ must also be projective of rank $d(M)$ by Proposition~\ref{Prop::General}. It follows via \cite[Chapter VII Proposition 4.10.24]{bourbaki} that $P_N \cong \Lambda^{d(M) - 1} \oplus \Ic_N$ for some invertible ideal $\Ic_N \trianglelefteq \Lambda$ which determines unniquely the class $[\Ic_N] \in \operatorname{Cl}(\Lambda)$. Thus we obtain an isomorphism
    \begin{equation}\label{Eq::Dec}
        N \cong T_\Lambda(M) \oplus \Lambda^{d(M) - 1} \oplus \Ic_N
    \end{equation}
    noting further that the isomorphism type of the invertible ideal $\Ic_N$ must determine the isomorphism type of the $\Lambda$-module $N$ uniquely. Hence the assignment
    \[
    \mathcal{Y}_{\Lambda,M}(N) = [\Ic_N]
    \]
    yields a well-defined function as in (\ref{Eq::koten}). Moreover, given two $\Lambda$-modules $N_1, N_2 \in \Gs[\Lambda]{N}$ with $\mathcal{Y}_{\Lambda,M}(N_1) = \mathcal{Y}_{\Lambda,M}(N_2)$, the associated ideals $\Ic_{N_1}$ and $\Ic_{N_2}$ belong to the same class and must hence be isomorphic. The isomorphim (\ref{Eq::Dec}) then yields an isomorphism of $\Lambda$-modules $N_1 \cong N_2$. Thus $\mathcal{Y}_{\Lambda,M}$ is an injective assignment and the proof is complete.
\end{proof}
If $\Lambda = \mathcal{O}_K$ is the ring of integers in an algebraic number field $K$, then $\Lambda$ is a Dedekind domain \cite[Proposition I.8.1]{ANT} and its ideal class group is finite \cite[Theorem I.6.3]{ANT}. Hence we obtain the following corollary to Theorem~\ref{Thm::Dedekind}.
\begin{corollary}\label{Cor::NumberField}
Let $\mathcal{O}_K$ be the ring of integers in an algebraic number field $K$ and $M$ be a finitely generated $\mathcal{O}_K$-module. Then $\Gs[\mathcal{O}_K]{M}$ is finite.
\end{corollary}

On the other hand, principal ideal domains are precisely Dedekind domains with trivial ideal class group. Thus we obtain the following absolute profinite rigidity result for arbitrary finitely generated modules over a principal ideal domain.

\corB*

\section{Applications to the profinite rigidity of groups}\label{Sec::Gps}
In this section, we shall apply the results developed in Sections \ref{Sec::PI} and \ref{Sec::PIDs} to the profinite rigidity of groups. The salient result is Theorem~\ref{Thm::Groups}---the absolute profinite rigidity of solvable Baumslag--Solitar groups---proven in subsection Section~\ref{Sec::BS} below. We refer the reader also to parallel work of the author \cite{FreeMetab} where the results of the present article have been used to prove the absolute profinite rigidity of free metabelian groups, as well as ongoing work with further classes of solvable groups.

The idea is outlined as follows. Suppose that $\Gamma$ is a finitely generated group with (possibly infinitely generated) abelian normal subgroup $N \trianglelefteq \Gamma$ and quotient $\Upsilon = \Gamma / N$. The subgroup $N$ acquires canonically the structure of a module over the group algebra $\Z[\Upsilon]$ whereby $\Upsilon$ acts by conjugation. We shall refer to this module structure on $N$ as the \emph{$\Gamma$-conjugacy structure}. One may then relate the question of profinite rigidity of $\Gamma$ to the profinite rigidity of $N$ as a module over the ring $\Lambda = \Z[\Upsilon]$. An essential assumption will be that $\Upsilon$ \emph{has separable cohomology}, meaning that the profinite completion map $\iota \colon \Upsilon \to \widehat{\Upsilon}$ induces an isomorphism on cohomology
\[
\iota^\ast \colon H^\ast(\widehat \Upsilon,M) \xrightarrow{\sim} H^*(\Upsilon, M)
\]
for any finite $\Z[\Upsilon]$-module $M$. This concept---also known as $\Upsilon$ being \emph{cohomologically good}---was first introduced by Serre in \cite[Section 2]{Serre} and has recently gained considerable new attention: we invite the reader to \cite[Chapter 7]{Gareth_Book} or \cite{WillAndI} for details. In the presence of separable cohomology, we obtain the following lemma.
\begin{lemma}\label{Lem::ConjugacyModules}
    Let $\Gamma$ be a finitely generated residually finite group, $M \trianglelefteq \Gamma$ an abelian normal subgroup and $\Upsilon = \Gamma / M$ the associated quotient. Consider $M$ as a module over $\Lambda = \Z[\Upsilon]$ with $\Gamma$-conjugacy structure. If $\Upsilon$ has separable cohomology in dimension two, then the $\Lambda$-profinite completion $\hM$ is $\Lambda$-isomorphic to the profinite $\Lambda$-module $\overline{M}$ given by the closed subgroup $\overline{M} \leq \hG$ with $\hG$-conjugacy structure.
\end{lemma}
\begin{proof}
    The $\Lambda$-module $\hM$ is given by the inverse limit of the inverse system $\mathcal{I}_\Lambda(M)$ of $\Lambda$-epimorphic images of the $\Lambda$-module $M$. On the other hand, the $\Lambda$-module $\overline{M}$ is given by the inverse system $\mathcal{I}_\Gamma(M)$ of the images $\pi(M)$ of the subgroup $M \leq \Gamma$ under finite quotients $\pi \colon \Gamma \twoheadrightarrow Q$, whose $\Lambda$-module structures extend $\Z$-linearly from the conjugation action induced via $\Upsilon \twoheadrightarrow Q/\pi(M)$. Thus it will suffice to show that these inverse systems are isomorphic. Indeed, consider the natural embedding
    \[
    F \colon \mathcal{I}_\Gamma(M) \to \mathcal{I}_\Lambda(M)
    \]
   which realises the restriction of a finite quotient $\pi \colon \Gamma \twoheadrightarrow Q$ to $M$ as a finite quotient of abelian groups $F(\pi) = \pi \at{M} \colon M \twoheadrightarrow Q_M \leq Q$ that is $\Lambda$-equivariant as the induced $\Lambda$-action on $Q_M$ factors through conjugation in $Q$. To see that $F$ has cofinal image, choose any epimorphism of $\Lambda$-modules $p \colon M \twoheadrightarrow Q_M$ with finite codomain and write $K = \operatorname{Ker}(p)$ for its kernel, which forms a normal subgroup of $M$. The action of $\Lambda$ on $M$ is then given by conjugation in $\Gamma$, so the assumption that $p$ is a morphism of $\Lambda$-modules implies that $K$ is also a normal subgroup in $\Gamma$. Hence we obtain an extension
    \[
    1 \to Q_M \to \Gamma / K \to \Upsilon \to 1
    \]
    of the finitely generated group $\Upsilon$ by the finite group $Q_M$. By assumption, $\Upsilon$ has separable cohomology in dimension two, so \cite[Theorem 7.2.6]{Gareth_Book} yields a finite-index normal subgroup $\widetilde{U} \trianglelefteq \Gamma/K$ with $\widetilde{U} \cap Q_M = 1$. Now the preimage $U$ of $\widetilde{U}$ in $\Gamma$ is a finite-index normal subgroup of $\Gamma$ and one obtains an epimorphism of groups $\pi \colon \Gamma \to \Gamma / U$ fitting into the commutative diagram
    \[
    \begin{tikzcd}
        1 \arrow[r] & M \arrow[r] \arrow[d, "p", two heads] & \Gamma \arrow[r, "\phi"] \arrow[d, two heads, "\pi"] & \Upsilon \arrow[r] \arrow[d, two heads] & 1 \\
        1 \arrow[r] & Q_M \arrow[r]& \Gamma / U \arrow[r] & \frac{\Gamma}{U\cdot M} \cong \frac{\Upsilon}{\phi(U)} \arrow[r] & 1
    \end{tikzcd}
    \]
    whose rows are short exact sequences. The epimorphism of groups $\pi$ then restricts to $p$ on $M$. We have shown that the embedding of inverse systems $F$ has cofinal image; applying the inverse limit functor, we obtain an isomorphism of profinite $\Lambda$-modules \[\hM = \lim_{\longleftarrow}(\mathcal{I}_\Lambda(M)) \cong \lim_{\longleftarrow}(\mathcal{I}_\Gamma(M)) = \overline{M}\] as postulated.
\end{proof}

\subsection{Solvable Baumslag--Solitar Groups} \label{Sec::BS}
A class of groups of particular significance in combinatorial and geometric group theory is the class of \emph{Baumslag--Solitar groups}. These groups, parametrised by pairs of non-zero integers $(n,m) \in \Z^2$, are given by the two-generator one-relator presentation
\[
\Gamma_{m,n} = \operatorname{BS}(m,n) = \langle a,t \mid t^{-1}a^nt = a^m \rangle
\]
where, after possibly composing with $\Gamma$-automorphisms, we may assume without loss of generality that $|n| \geq m \geq 0$ holds.  Equivalently, Baumslag--Solitar groups are precisely HNN-extensions of the infinite cyclic group along a non-trivial edge group (whereby one recovers $n,m$ as the images of the edge group generator under the prescribed inclusion). We refer the reader to \cite{BSS1, BSS3} for a survey of the properties and significance of this class of groups. Most notably, Baumslag and Solitar exhibited in 1962 the group $\operatorname{BS}(2,3)$ as the first example of a finitely generated one-relator group which is non-Hopfian, and therefore, not residually finite \cite{BS}. At present, it is known that $\operatorname{BS}(n,m)$ is residually finite if and only if either $|m| = |n|$ or else $m = 1$ holds; it is subgroup separable (LERF) if and only if $|m| = |n|$ holds \cite[Theorem 1 and Theorem 13]{BSS1}.

Owing to their peculiar residual properties, Baumslag--Solitar groups have also been of particular interest within profinite rigidity. For instance, the cohomological separability of Baumslag--Solitar groups was characterised completely in \cite{WillAndI}, while \cite{Wang} has established relative profinite rigidity of residually finite Baumslag--Solitar with respect to one another. From the methods in the present paper we obtain the \emph{absolute} profinite rigidity of $\operatorname{BS}(1,n)$ for any integer $n \in \mathbb{Z}$, i.e. that these groups can be distinguished by their profinite completions among \emph{all finitely generated residually finite groups}, as postulated in Theorem~\ref{Thm::Groups}. This subclass coincides precisely with the class of solvable Baumslag--Solitar groups, and constitutes the first instance of absolute profinite rigidity among non-abelian one-relator groups, as well as the first instance among groups which are not subgroup separable (LERF). We dedicate the remainder of this section to the proof of Theorem~\ref{Thm::Groups}, which utilises the full strength of the theory developed in the present article.

\thmB*

\begin{proof}
    Let $\Gamma = \operatorname{BS}(1,n)$ be the solvable Baumslag--Solitar group associated to an integer $n \in \mathbb{Z}$, that is, the group given by the presentation
    \begin{equation}\label{Eq::PresBS}
    \Gamma = \langle a,t \mid t^{-1}at = a^n \rangle \cong \Z\left[\tfrac{1}{n}\right] \rtimes_n \Z
    \end{equation}
    which is residually finite by \cite[Theorem 1]{BSS1}. We may assume that $n \neq 0,1$, since the cases $n = 0$ and $n = 1$ yield abelian groups already known to be profinitely rigid in the absolute sense \cite[Corollary 3.2.12]{Gareth_Book}. Let $\Delta$ be a finitely generated residually finite group admitting an isomorphism of profinite completions $\phi \colon \hG \xrightarrow{\sim} \hD$. By \cite[Proposition 3.2.10]{Gareth_Book}, we must have
    \begin{equation}\label{Eq::BSAbel}
        \Delta^\mathrm{ab} \cong \Gamma^\mathrm{ab} \cong \frac{\Z}{n-1} \times \Z
    \end{equation}
    so that there are epimorphisms of groups $\eta_\Gamma \colon \Gamma \twoheadrightarrow \Z$ and $\eta_\Delta \colon \Delta \twoheadrightarrow \Z$ that are unique up to multiplication by $-1$. Let $M = \operatorname{Ker}(\eta_\Gamma) \trianglelefteq \Gamma$ and $N = \operatorname{Ker}(\eta_\Delta) \trianglelefteq \Delta$ be the kernels of these maps, which acquire canonically the structure of a module over the group algebra $\Omega = \Z[\Z] \cong \Z[x^\pm]$, where $x$ acts as conjugation by a preimage of the generator of $\Z$ under $\eta_\Gamma$ or $\eta_\Delta$, respectively. Observe that $M = \langle \langle a \rangle \rangle \trianglelefteq \Gamma$ takes the form
    \begin{equation}\label{Eq::PresM}
    M \cong \frac{\Omega}{(x-n)}
    \end{equation}
    as derived from the presentation (\ref{Eq::PresBS}). Now
    \begin{equation}\label{Eq::BSAbel}
        \hD^\mathrm{ab} \cong \hG^\mathrm{ab} \cong \frac{\Z}{n-1} \times \pZ
    \end{equation}
    so the induced maps $\widehat{\eta_\Gamma} \colon \hG \twoheadrightarrow \pZ$ and $\widehat{\eta_\Delta} \colon \hD \twoheadrightarrow \pZ$ are also unique, up to multiplication by a unit in $\pZ^\times$. Moreover, the conjugacy $\widehat{\Omega}$-module structure on the closed subgroups $\overline{M} = \operatorname{Ker}(\widehat{\eta_\Gamma})\trianglelefteq \hG$ and $\overline{N} = \operatorname{Ker}(\widehat{\eta_\Delta})\trianglelefteq \hD$ agree with the $\Omega$-profinite completions $\hM$ and $\hN$ of $M$ and $N$, respectively, as per Lemma~\ref{Lem::ConjugacyModules}. Hence we obtain a commutative diagram
    \begin{equation}\label{Eq::OGCD}       
    \begin{tikzcd}
1 \arrow[r] & \hM \arrow[d, "f"] \arrow[r] & \hG \arrow[d, "\phi"] \arrow[r] & \pZ \arrow[d, "\cdot \kappa"] \arrow[r] & 1 \\
1 \arrow[r] & \hN \arrow[r]                & \hD \arrow[r]                   & \pZ \arrow[r]                           & 1
\end{tikzcd}
    \end{equation}
    where $\kappa \in \pZ^\times$ is a unit and $f$ is an isomorphism of profinite abelian groups. Fruthermore, multiplication by $\kappa$ extends $\pZ$-linearly to an automorphism of profinite rings $\alpha \colon \widehat{\Omega} \xrightarrow{\sim} \widehat{\Omega}$ which satisfies the twisted equation
    \begin{equation}\label{Eq::Twist}
    f(\omega \cdot m) = \alpha(\omega) \cdot f(m)
    \end{equation}
    for any $\omega \in \widehat{\Omega}$ and $m \in \hM$. We shall write $\Lambda_M$ for the ring $\Lambda_M = \frac{\Omega}{\operatorname{Ann}_{\Omega}(M)} \cong \frac{\Omega}{(x-n)}$ so that the $\Omega$-module structure on $M$ is equivalent to a rank-one free $\Lambda_M$-module structure. Similarly, we shall write $A$ for the annihilator $A = \Ann[\Omega]{N}$ in $\Omega$ and $\Lambda_N$ for the quotient ring $\Lambda_N = \Omega / A$, so that the $\Omega$-module structure on $N$ is equivalent to a $\Lambda_N$-module structure with trivial annihilator. We proceed to narrow down the possible structure of $A$ and $N$ in three steps.

    \textbf{Step I.} \textit{The annihilator $A = \Ann[\Omega]{N}$ is a principal ideal in $\Omega$, and for each maximal ideal $\Ic \trianglelefteq \Omega$, the localisation $A_\Ic$ is a non-maximal prime ideal in $\Omega_\Ic$}.

    We show first that the local statement implies the global. The ideal $A$ in the domain $\Omega$ is principal if and only if it is free of rank one. On the other hand, owing to a celebrated result of Quillen and Suslin \cite[Corollary 7.4]{suslin}, the localised polynomial ring $\Omega$ is homologically taut, so being free of rank one is in turn equivalent to being projective of rank one as a module over $\Omega$. Finally, rank-one projectivity of the ideal $A$ is equivalent to the property that the localisation $A_\Ic \trianglelefteq \Omega_\Ic$ is principal for each maximal ideal $\Ic \trianglelefteq \Omega$, as per \cite[Chapter II Theorem 5.6.4]{bourbaki}. That the latter holds whenever $A_\Ic$ is non-maximal prime follows from the fact that all prime ideals in $\Omega$ are either maximal or principal \cite[Example II.1.H]{RedBook}. Hence, the local statement does indeed imply the global, and it will suffice to show that for any maximal ideal $\Ic \trianglelefteq \Omega$, the localisation $A_\Ic$ is a non-maximal prime ideal in $\Omega_\Ic$.

    Choose a maximal ideal $\Ic \trianglelefteq \Omega$ and consider the diagram induced by the $\Ic$-adic completion functor
    \[
    \begin{tikzcd}
1 \arrow[r] & A_\Ic \arrow[d] \arrow[r]   & \Omega_\Ic \arrow[d] \arrow[r]   & \Omega_\Ic/A_\Ic \arrow[r] \arrow[d]               & 1 \\
1 \arrow[r] & A\ad \arrow[r] & \Omega\ad \arrow[r] & \Omega\ad/A\ad \arrow[r] & 1
\end{tikzcd}
    \]
    which commutes as $\Ic$-adic completion is exact \cite[Chapter III Theorem 3.4.3]{bourbaki}. As the ring $\Omega_\Ic/A_\Ic$ is local, the map $\Omega_\Ic/A_\Ic \to \Omega\ad/A\ad$ must be injective \cite[Chapter III Proposition 3.3.5]{bourbaki}, so in fact it will suffice to prove that the $\Ic$-adic completion $\Omega\ad/A\ad$ is a domain, or equivalently, that the $\Ic$-adic completion $A\ad$ of the ideal $A_\Ic$ is a non-maximal prime ideal in $\Omega\ad$. Indeed, the automorphism of profinite rings $\alpha \colon \widehat \Omega \xrightarrow{\sim} \widehat \Omega$ descends to an isomorphism of $\Ic$-adic completions
    \[
    \begin{tikzcd}
\widehat\Omega \arrow[r, "\alpha"] \arrow[d, two heads]    & \widehat\Omega \arrow[d, two heads] \\
\Omega_{\widehat{\Jc}} \arrow[r, "\alpha\ad"] & \Omega\ad        
\end{tikzcd}
    \]
    where $\Jc \trianglelefteq \Omega$ is the maximal ideal $\Jc = \Omega \cap \alpha^{-1}(\Ic)$, cf. \cite[Lemma 2.1]{jaikin_genus}. Consider hence the map
    \[
    F_\Ic \colon M_{\widehat{\Jc}} \cong \hM \otimes_{\widehat{\Omega}} \Omega_{\widehat{\Jc}} \xrightarrow{f \otimes \alpha_{\Ic}} \hN \otimes_{\widehat{\Omega}} \Omega\ad \cong N\ad
    \]
    where the first and third isomorphisms are given by Lemma~\ref{Prop::Localisation}. To see that $F_\Ic$ is a well-defined isomorphism of profinite abelian groups, note that for any 
    \begin{align*}
    F_\Ic(\lambda m \otimes \omega) &= \alpha(\lambda)f(m) \otimes \alpha\ad(\omega) \\
    &= f(m) \otimes \alpha(\lambda) \alpha\ad(\omega)\\
     &= F_\Ic(m \otimes \lambda\omega) 
    \end{align*}
    holds whenever $\lambda,\omega \in \Omega_{\widehat{\Jc}}$ and $m \in M_{\widehat{\Jc}}$. For the same reason, $F_\Ic$ forms a twisted isomorphism of profinite $\widehat{\Omega}$-modules in the sense of (\ref{Eq::Twist}). It follows that
    \[
    A\ad = \Ann[\Omega\ad]{N\ad} = \alpha_{\widehat\Ic}\left(\Ann[\Omega_{\widehat\Jc}]{M_{\widehat\Jc}}\right) = \alpha_{\widehat\Ic}\left((x-n)_{\widehat\Jc}\right)
    \]
     where the first equality holds by \cite[Chapter III Corollary 3.4.1]{bourbaki} and the second derives from the twist (\ref{Eq::Twist}). As the ideal $(x-n) \trianglelefteq \Omega$ is non-maximal, neither is $(x-n)_{\widehat\Jc}$, see \cite[Chapter III Proposition 3.4.8]{bourbaki}. To see that it is prime, note that the ring $\frac{\Omega_\Jc}{(x-n)_\Jc}$ is regular local (it is the localisation of a principal ideal domain), so its completion is regular local as well \cite[Exercise 19.1]{eisenbud} and hence a domain \cite[Corollary 10.14]{eisenbud}. We conclude that $A_{\widehat \Ic} = \alpha_{\widehat \Ic} \left( (x-n)_{\widehat{\Jc}}\right)$ is a non-maximal prime ideal and the proof of Step I is complete.
     
     \textbf{Step II.} \textit{The annihilator $A = \Ann[\Omega]{N}$ is of the form $A = (ax - b)$ for some coprime integers $a,b \in \Z$}.

     By Step I, the annihilator $A = \Ann[\Omega]{N}$ is a principal ideal, i.e. $A = (g(x))$ for some Laurent polynomial $g(x) \in \Omega = \Z[x^\pm]$. After possibly multiplying by a unit in $\Omega^\times$, we may assume that $g$ has non-negative degree and non-zero constant term. Identifying $A$ with its image in $\widehat{\Omega}$ and writing $\overline{A}$ for its closure in the profinite topology on $\widehat{\Omega}$, we find that 
     \[
     \overline{(g(x))} = \overline{\Ann[\Omega]{N}} = \Ann[\widehat\Omega]{\hN} = \alpha\left(\Ann[\widehat{\Omega}]{\hM}\right) = \alpha \left(\overline{(x-n)}\right)
     \]
     holds in $\widehat \Omega$, where the second and fourth equalities are given by Proposition~\ref{Prop::AnnAndComp} and the third derives from the twist (\ref{Eq::Twist}). It follows that the isomorphism of profinite rings $\alpha \colon \widehat{\Omega} \xrightarrow{\sim} \widehat{\Omega}$ descends to an isomorphism
     \[
     \begin{tikzcd}
\widehat\Omega \arrow[r, "\alpha"] \arrow[d, two heads]                                   & \widehat\Omega \arrow[d, two heads]                          \\
\widehat{\Lambda_M} \cong \widehat{\Omega}/\overline{(x-n)} \arrow[r, "\widetilde\alpha"] & \widehat{\Omega}/\overline{(g(x))} \cong \widehat{\Lambda_N}        
\end{tikzcd}
     \]
as profinite rings. Choose a prime number $p$ which divides neither $n$ nor any of the coefficients of $g(x)$. The principal ideal $(p)$ is characteristic in the profinite ring $\widehat{\Lambda_M}$ (see e.g. \cite[Chapter III Proposition 2.12.16]{bourbaki}) and is hence preserved by the isomorphism of profinite rings $\widetilde{\alpha} \colon \widehat{\Lambda_M} \xrightarrow{\sim} \widehat{\Lambda_N}$. Thus we find that the profinite completion
\[
\widehat{\frac{\Omega}{(p,g(x))}} \cong \widehat{\frac{\Lambda_N}{(p)}} \cong \widehat{\frac{\Lambda_M}{(p)}} \cong \widehat{\frac{\Omega}{(p,x-n)}} \cong \mathbb{F}_p
\]
is the finite field of order $p$, and in fact there must be an isomorphism of the discrete ring $\frac{\Omega}{(p,g(x))} \cong \mathbb{F}_p$. It follows that the reduction of $g(x)$ modulo $p$ is linear and irreducible. But $p$ does not divide any of the coefficients of $g(x)$, so the polynomial $g(x)$ must itself be linear and irreducible in $\Omega$. In other words, we have shown that $g(x) = ax - b$ for two non-zero integers $a,b \in \Z$. Moreover, the greatest common divisor $d := \gcd(a,b)$ must satisfy $d = 1$, for otherwise there would exist a prime number $q$ dividing $d$ and the localisation of $A = (ax - b)$ at the maximal ideal $(q,\frac{a}{d}x-\frac{b}{d}) \trianglelefteq \Omega$ would not be a prime ideal, contradicting Step I. Thus $A = (ax - b)$ for two coprime integers $a,b \in \mathbb{Z}$, and Step II is established.

\textbf{Step III.} \textit{The group $\Delta$ is of the form $\Delta \cong \Z\left[\tfrac{1}{ab}\right] \rtimes_{b/a} \Z$ for two coprime integers $a,b \in \Z$. Moreover, a prime number $p$ divides $n$ if and only if $p$ divides $ab$.}

Indeed, Step II allows us to realise the ring $\Lambda_N$ as the domain
\[
\Lambda_N = \frac{\Omega}{\Ann[\Omega]{N}} = \frac{\Omega}{(ax-b)} \cong \Z\left[\tfrac{1}{ab}\right]
\]
where $a,b \in \mathbb{Z}$ are coprime integers. Consider the isomorphism of profinite abelian groups given by the composition
\[
F \colon \widehat{\Lambda_N} \xrightarrow{\widetilde\alpha^{-1}} \widehat{\Lambda_M} \cong \hM \xrightarrow{f} \hN
\]
where $\widetilde\alpha$ is the isomorphism of profinite rings constructed in Step II, the central isomorphism of $\Lambda_M$-modules $\widehat{\Lambda_M} \cong \hM$ derives from the construction of $M$ as a rank-one free $\Lambda_M$-module, and $f$ is the twisted isomorphism of $\Omega$-modules described in (\ref{Eq::Twist}). Observe that

\[
F(\lambda \cdot \mu) = f(\widetilde\alpha^{-1}(\lambda) \cdot \widetilde\alpha^{-1}(\mu)) = \widetilde\alpha(\widetilde\alpha^{-1}(\lambda)) \cdot f(\widetilde\alpha^{-1}(\mu)) = \lambda \cdot F(\mu)
\]
holds for all $\lambda, \mu \in \widehat{\Lambda_N}$. Thus $F$ is in fact an isomorphism of profinite $\Lambda_N$-modules. As $\Lambda_N \cong \Z[\tfrac{1}{ab}]$ is the localisation of a principal ideal domain, it is a principal ideal domain itself \cite[Proposition 3.11]{AM}. We may then use either Theorem~\ref{Thm::epic} or Corollary~\ref{Cor::PID} to deduce that there exists an isomorphism of discrete $\Lambda_N$-modules $N \cong \Lambda_N$. Hence we obtain a short exact sequence of groups
\[
1 \to N \cong \Z\left[\tfrac{1}{ab}\right] \xrightarrow{\,\,\,\,\,\,} \Delta \xrightarrow{\eta_\Delta} \Z \to 1
\]
which splits (as $\Z$ is a free group), yielding
\begin{equation}\label{Eq::StructureOfDelta}
\Delta = N \rtimes_x \Z \cong \Z\left[\tfrac{1}{ab}\right] \rtimes_{b/a} \Z
\end{equation}
as postulated. Moreover, we find that
\[
\prod_{p \nmid n} \Z_p \rtimes_{n} \pZ \cong \hG \cong \hD \cong \prod_{p \nmid ab} \Z_p \rtimes_{b/a} \pZ 
\]
as per \cite[Example 4.1]{WillAndI}, wherefrom it follows that a prime number $p$ divides $n$ if and only if $p$ divides $ab$. This proves Step III.

\textbf{Conclusion.} \textit{There is an isomorphism of discrete groups $\Gamma \cong \Delta$.}

Indeed, consider the semidirect product decompositions of $\Gamma$ and $\Delta$ given in (\ref{Eq::PresBS}) and (\ref{Eq::StructureOfDelta}), respectively. Writing the kernels $M$ and $N$ additively and the $\Z$-factors $\eta_\Gamma(\Gamma) = \langle t_\Gamma \rangle \cong \Z$ and $\eta_\Delta(\Delta) = \langle t_\Delta \rangle \cong \Z$ multiplicatively, we find that
\[
(0,t_\Gamma) \ast_{\hG} (1,1) = (t_\Gamma \cdot 1, t_\Gamma) = (n,t_\Gamma)
\]
holds in $\hG$. Passing through the isomorphism of profinite groups $\phi \colon \hG \to \hD$, we infer that
\begin{align*}    
(n \cdot f(1),t^\kappa_\Delta) &= \phi(n,t_\Gamma) \\
&= \phi\left( (0,t_\Gamma) \ast_{\hG} (1,1)\right) \\
&= \phi(0,t_\Gamma) \ast_{\hD} \phi(1,1) \\
&= (0,t_\Delta^\kappa) \ast_{\hD} (f(1),1) \\
&= \left(\left( \tfrac{b}{a}\right)^\kappa \cdot f(1), t_\Delta^\kappa \right)
\end{align*}
holds in $\hD$, where $\kappa \in \pZ^\times$ and $f \colon \hM \to \hN$ are the maps given in (\ref{Eq::OGCD}). But $f(1)$ cannot be a zero-divisor in the ring $\widehat{\Lambda_N} \cong \prod_{p \nmid ab} \Z_p$, so we must in fact have
\begin{equation}\label{Eq::nab}
    n = (b/a)^\kappa
\end{equation}
in $\widehat{\Lambda_N}$. Now, for any prime $p$ which does not divide the product $ab$, there exists an epimorphism of groups of units 
\[
\zeta_p \colon \widehat{\Lambda_N}^\times \cong \prod_{p \nmid ab} \Z_p^\times \twoheadrightarrow \Z_p^\times \twoheadrightarrow \mathbb{F}_p^\times
\]
given by projections. The equation (\ref{Eq::nab}) then implies that there is an inclusion of subgroups $\langle \zeta_p(b/a)\rangle \subseteq \langle \zeta_p(n)\rangle $ in the group of units $\mathbb{F}_p^\times$. As $\kappa \in \mathbb{Z}^\times$ is a unit itself, we also obtain the opposite inclusion and hence the equality $\langle \zeta_p(b/a)\rangle = \langle \zeta_p(n)\rangle$. In particular, it follows that the rational numbers $n$ and $b/a$ have equal orders in the multiplicative group $\mathbb{F}_p^\times$ for almost all primes $p$. We now invoke the solution \cite[Theorem 1]{corrales_schoof} to a question of Erd\"os, which states that this can happen only if either $n = (b/a)^\pm$ or else $n = \pm 1$ and $b/a = \pm 1$ holds. But the integers $a$ and $b$ are coprime, so it must be the case that $a = \pm 1$ or $b = \pm 1$. After possibly composing with the automorphism of $\Delta$ induced by $t_\Delta \mapsto t_\Delta^{-1}$ or the automorphism of $\Delta$ induced by $1_N \mapsto -1_N$, we obtain $a = 1$ and $b = n$. Combining this data with the description of $\Delta$ established in Step III, we find that
\[
\Delta \cong \Z\left[\tfrac{1}{ab}\right] \rtimes_{b/a} \Z \cong \Z\left[\tfrac{1}{n}\right] \rtimes_{n} \Z \cong \Gamma
\]
is an isomorphism of discrete groups. Thus $\Gamma$ is profinitely rigid in the absolute sense and the proof is complete.

\end{proof}

\printbibliography

\end{document}